\documentclass[11pt,a4paper,reqno]{amsart}
\usepackage{amssymb,amsmath,amsthm}
\usepackage{graphicx}
\usepackage{url}

\theoremstyle{plain}
\newtheorem{theorem}{Theorem}
\newtheorem{lemma}[theorem]{Lemma}

\newcommand{\wet}[1]{w(#1)} 
\newcommand{\eps}{\varepsilon}
\newcommand{\N}{\mathbb{N}}
\newcommand{\R}{\mathbb{R}}
\newcommand{\E}{\mathbb{E}}
\newcommand{\hh}{\mathcal{H}}
\DeclareMathOperator \vol{Vol}
\DeclareMathOperator \conv{conv}

\newcommand{\pth}[1]{\left( #1 \right)}
\newcommand{\ptb}[1]{\left[ #1 \right]}

\begin{document}

\title[Random polytopes and the wet part]
{Random polytopes and the wet part for arbitrary probability distributions}

\author[B\'ar\'any, Fradelizi, Goaoc, Hubard, Rote]{Imre B\'ar\'any, Matthieu Fradelizi, Xavier Goaoc, Alfredo Hubard, G\"unter Rote}

\begin{abstract}
  We examine how the measure and the number of vertices of the
  convex hull of a random sample of an arbitrary probability measure
  in $\R^d$ relates to the wet part of that measure.
\end{abstract}

\maketitle

\section{Introduction and Main Results}

Let $K$ be a convex body (convex compact set with non-empty interior)
in~$\R^d$, and let $X_n=\{x_1,\ldots,x_n\}$ be a random sample of $n$
uniform independent points from $K$. The set $P_n=\conv X_n$ is
a \emph{random polytope} in $K$. For $t\in [0,1)$ we define
the \emph{wet part} $K_t$ of $K$:
\[\begin{aligned}
K_t=\{\,x \in K: & \mbox{ there is a halfspace }h  \mbox{ with } x \in h\\
& \mbox{ and }\vol(K\cap h) \le t\vol K\,\}\\
\end{aligned}
\]
The name ``wet part'' comes from the mental picture when $K$ is in
$\R^3$ and contains water of volume $t\vol K$. B\'ar\'any and
Larman~\cite{BarLar} proved that the measure of the wet part captures
how well $P_n$ approximates $K$ in the following sense:

\begin{theorem}[{\cite[Theorem~1]{BarLar}}] \label{th:BL}
  There are constants $c$ and $N_0$ depending only on $d$ such that
  for every convex body $K$ in $\R^d$ and for every $n>N_0$
  \[ \frac 14 \vol K_{1/n} \le \E[\vol(K\setminus P_n)] \le \vol K_{c/n}.\]
\end{theorem}
By Efron's formula
 (see \eqref{eq:Efron} below),
this directly translates into bounds for the expected number 
of vertices of~$P_n$, see Section~\ref{f-vectors}.

\subsection{Results for general measures.}

The notions of random polytope and wet part extend to a general
probability measure $\mu$ defined on the Borel sets of $\R^d$. The
definition of a $\mu$-random polytope $P_n^{\mu}$ is clear: $X_n$ is a
sample of $n$ random independent points chosen according to $\mu$, and
$P_n^{\mu}=\conv X_n$. The wet part $W_t^{\mu}$ is defined as
\[ W_t^{\mu}=\{\,x \in \R^d: \mbox{ there is a halfspace }h \mbox{
with } x \in h \mbox{ and }\mu (h) \le t\,\}.\]
The $\mu$-measure of the wet part is denoted by
$w^{\mu}(t):=\mu(W_t^{\mu})$. Here is an extension of Theorem~\ref{th:BL}
to general measures:

\begin{theorem}\label{th:BLgen}
  For any probability measure $\mu$ in $\R^d$ and $n \ge 2$,
   \[ \frac 14  w^{\mu}(\tfrac1n) \le \E[1-\mu(P_n^\mu)] \le 
     w^{\mu}({(d+2)\tfrac{\ln n}n}) + \tfrac{\varepsilon_d(n)}{n},\]
          where $\varepsilon_d(n)\to 0$ as $n\to+\infty$ and is independent of $\mu$.
\end{theorem}

\noindent
A similar upper bound, albeit with worse constants,
follows from a result of Vu~\cite[Lemma 4.2]{VanVu}, which states that
$P_n^\mu$ contains $\R^d \setminus W_{c\ln n/n}^{\mu}$ with
high-probability. Since a containment with high probability is usually
stronger than an upper bound in expectation, one may have hoped that the
$\log n/n$ in the upper bound of Theorem~\ref{th:BLgen} can be
reduced. Our main result shows that this is not possible, not even in
the plane:

\begin{theorem}\label{th:example}
  There exists a probability measure $\nu$ on $\R^2$ such that \[
    \E[1-\mu(P_n^\nu)] > \frac 1{2} \cdot
    w^\nu(\log_2 n/ n) \] for infinitely many $n$.
\end{theorem}

The measure that we construct actually has compact support and can be
embedded into $\R^d$ for any $d \ge 2$. It will be apparent from the
proof that the same construction has the stronger property
that for \emph{every} constant $C>0$, the inequality
$\E[1-\mu(P_n^\nu)] > \frac 1{2} \cdot w^\nu({C \log_2
n/ n})$ holds for infinitely many values~$n$.

\subsection{Consequences for $\mathbf{f}$-vectors}
\label{f-vectors}

 Let $f_0(P_n^\mu)$ denote the number of vertices of
$P_n^\mu$. For \emph{non-atomic} measures (measures where no single
point has positive probability), Efron's formula~\cite{Ef} relates
$E\ptb{f_0(P_n^\mu)}$ and $\E\ptb{\mu(P_n^\mu)}$:
%
  \begin{align} \label{eq:Efron0}
  \E [f_0(P_n^\mu)] & = \sum_{i=1}^n \Pr\ptb{x_i \notin \conv(X_n \setminus \{x_i\})}\\
  & = n\cdot\int_x \Pr\ptb{x \notin P_{n-1}^\mu}d\mu(x) =n(1-\E [\mu(P_{n-1}^\mu)]) \label{eq:Efron}
  \end{align}
%
 For \emph{any} measure, this still holds as an inequality:
 \begin{equation}
   \label{eq:Efron2}
   \E [f_0(P_n^\mu)] \ge \sum_{i=1}^n \Pr\ptb{x_i \notin \conv(X_n \setminus \{x_i\})} =n(1-\E [\mu(P_{n-1}^\mu)])
 \end{equation}
The measure that is constructed in Theorem~\ref{th:example} is non-atomic.
As a consequence, Theorems~\ref{th:BLgen} and~\ref{th:example} give the following bounds for the number of vertices:

\begin{theorem}\label{th:fvector}

\begin{itemize}
\item[(i)] For any non-atomic probability measure $\mu$ in $\R^d$, 
   \[ \frac 1e n w^{\mu}(\tfrac1n) \le \E\ptb{f_0(P_n^\mu)} \le 
     n w^{\mu}\bigl({(d+2)\tfrac{\ln n}n}\bigr) + \varepsilon_d(n),\]
     where $\varepsilon_d(n)\to 0$ as $n\to+\infty$ and is independent of $\mu$.

\item[(ii)] There exists a non-atomic probability measure $\nu$ on $\R^2$ such that \[
     \E\ptb{f_0(P_n^\nu)} > \frac 1{2} n \cdot
    w^\nu(\log_2 n/ n) \] for infinitely many $n$.
\end{itemize}
\end{theorem}

\noindent
Theorem~\ref{th:fvector} follows from Theorems~\ref{th:BLgen} and~\ref{th:example} except
that Efron's Formula~\eqref{eq:Efron} induces a shift in indices, as
it relates $f_0(P_n^\mu)$ to $\mu(P_{n-1}^\mu)$. This shift
affects only the constant in the lower bound of
Theorem~\ref{th:fvector}(i), which goes from $\frac14$ to
$\frac1e$, see Section~\ref{p:lowerbound}.

The upper bound of Theorem~\ref{th:fvector}(i) fails for general
distributions. For instance, if $\mu$ is a discrete distribution on a
finite set, then $w^\mu(t)=0$ for any $t$ smaller than the mass of any
single point and the upper bound cannot hold uniformly as
$n \to \infty$. Of course, in that case
Inequality~\eqref{eq:Efron2} is strict.

\bigskip

For convex bodies, the number $f_i(P_n)$ of $i$-dimensional faces of
$P_n$ can also be controlled via the measure of the wet part since
B\'ar\'any~\cite{Bar} proved that $\E[f_i(P_n)] = \Theta(n \vol
K_{1/n})$ for every $0 \le i \le d-1$. No similar generalization is
possible for Theorem~\ref{th:BLgen}. Indeed, consider a measure $\mu$
in $\R^4$ supported on two circles, one on the $(x_1,x_2)$-plane, the
other in the $(x_3,x_4)$-plane, and uniform on each circle; $P_n^\mu$
has $\Omega(n^2)$ edges almost surely.

\bigskip

Before we get to the proofs of Theorems~\ref{th:BLgen}
(Section~\ref{s:epsnet}) and~\ref{th:example}
(Section~\ref{sec:example}), we discuss in Section~\ref{s:regularity}
a key difference between the wet parts of convex bodies and of general
measures.

\section{Wet part: convex sets versus general measures}\label{s:regularity}

A key ingredient in the proof of the upper bound of
Theorem~\ref{th:BL} in~\cite{BarLar} is that for a convex body $K$ in
$\R^d$, the measure of the wet part $K_t$ cannot change too abruptly
as a function of $t$: If $c\ge 1$, then
\begin{equation}\label{eq:wet}
\vol K_t \le \vol K_{ct} \le c' \vol K_t
\end{equation}
where $c'$ is a constant that depends only on $c$ and
$d$~\cite[Theorem~7]{BarLar}. In particular, a multiplicative factor
can be taken out of the volume parameter of the wet part and the upper
bound in Theorem~\ref{th:BL} can be equivalently expressed as
\begin{equation}\label{eq:alternative}
 \E[\vol(K\setminus P_n)] \le  c' \vol K_{1/n}. 
\end{equation}
(This is 
in fact
how of the upper bound
of Theorem~\ref{th:BL} is 
actually
formulated in~\cite[Theorem~1]{BarLar}.)
This alternative formulation shows immediately that the lower bound of
Theorem~\ref{th:BL} (and hence also of Theorem~\ref{th:BLgen}) cannot
be improved by more than a constant.

\subsection{Two circles and a sharp drop}

The right inequality in (\ref{eq:wet}) does not extend to general
measures.  An easy example showing this is the following ``drop
construction''. It is a probability measure~$\mu$ in the plane
supported on two concentric circles, uniform on each of them, and with
measure $p$ on the outer circle. Let~$\tau$ denote the measure of
a halfplane externally tangent to the inner circle; remark that $\tau<p/2$. The measure
$w^\mu(t)$ of the wet part drops at $t=\tau$:
\begin{equation}\label{eq:one-drop}
   w^\mu(t) = 
    \begin{cases}
      p, & \text{if } t < \tau\\
      1, & \text{if } t \ge \tau
    \end{cases}
\end{equation}
We can make this drop arbitrarily sharp by choosing a small $p$. In
particular, for any given $c'$, setting $p < \frac1{c'}$ makes it
impossible to fulfill the right inequality in (\ref{eq:wet}) for
$t<\tau<ct$.

\begin{figure}
  \centering \includegraphics{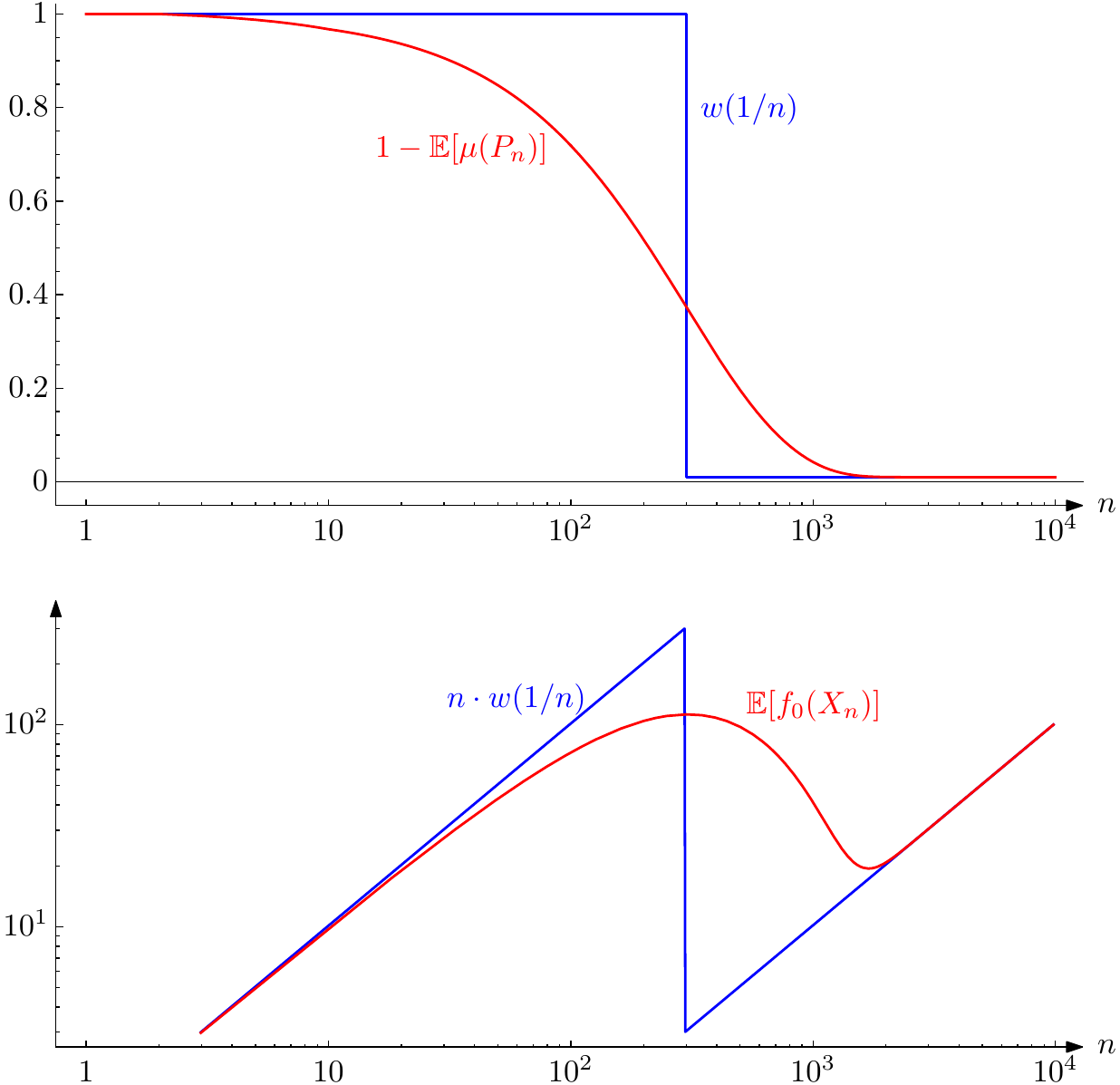}

 \caption{The quantities involved in Theorems~\ref{th:BL}--\ref{th:fvector}
for the drop construction with $p=1/100$, when the outer circle has twice
the radius of the inner circle. Top: $\E[1-\mu(P_n^\mu)]$ and $w(1/n)$, the $x$-axis being a logarithmic scale. Bottom: $\E[f_0(P_n^\mu)]$ and $n\cdot w(1/n)$ on a doubly-logarithmic scale.}
  \label{fig:example}
\end{figure}

This example also challenges
Inequality~\eqref{eq:alternative}. As shown in
Figure~\ref{fig:example} (top), the function $w^\mu(1/n)$ has a sharp drop,
while $\E\ptb{1-\mu(P_n^\mu)}$ shifts from the higher to the lower
branch of the step in a gradual way. For this construction, the
straightforward extension of Theorem~\ref{th:BL} would imply that
$\E[1-\mu(P_n^\mu)]$ remains within a constant multiplicative factor
of $ w^\mu(1/n)$. Thus, $\E[1-\mu(P_n^\mu)]$ would have to follow the
steep drop.

\subsection{A drop for the number of vertices.}

The fact that $\E[1-\mu(P_n^\mu)]$ {cannot} drop too
sharply is more easily seen by examining
$\E\ptb{f_0(P_n^\mu)}$. Since the measure defined in
Equation~\eqref{eq:one-drop} is non-atomic, Efron's
Formula~\eqref{eq:Efron} applies, so let us compare
$\E\ptb{f_0(P_n^\mu)}$ and $n\cdot n w^\mu(1/n)$. As illustrated
in Figure~\ref{fig:example} (bottom), $n\cdot w^\mu(1/n)$ has a
sawtooth shape with a sharp drop from 300 to~3 at $n=300$, and
$\E[f_0(P_n)]$ does actually shift from the higher to the lower branch
of the sawtooth, in a gradual way.

The fact that $\E[f_0(P_n^\mu)]$ can \emph{decrease} is perhaps
surprising at first sight, but this phenomenom is easy to explain: We
pick random points one by one. As long as all points lie on the inner
circle, $f_0(P_n^\mu)=n$. The first point to fall on the outer circle
swallows a constant fraction of the points into the interior of $P_n^\mu$,
while adding only a single new point on the convex hull, causing a big
drop. This happens around $n\approx 1/p$.

Again, the straightforward extension of Theorem~\ref{th:BL} would
imply that $\E[f_0(P_n)]$ follows the steep drop. Yet, on average, a
single additional point can reduce $f_0(P_n)$ by {a factor} of at
most~$1/2$. Hence, the drop of $\E[f_0(P_n)]$ cannot be so abrupt as
the drop of $n\cdot w^\mu(1/n)$, for $p$ small enough.

\subsection{A sequence of drops}

We prove Theorem~\ref{th:example} in Section~\ref{sec:example} by an
explicit construction that sets up a sequence of such drops. The
function $n\cdot w^\mu(1/n)$ reaches larger and larger peaks as $n$
increases, while dropping down more and more steeply between those
peaks. Our proof of Theorem~\ref{th:example} will not actually refer
to any drop or oscillating behavior.  We will simply identify a
sequence of values $n=n_1, n_2, \ldots$ for which
$\E[1-\mu(P_{n}^\mu)]$ is larger than $\frac12 w^\mu(\log_2
n/n)$.

\subsection{Open questions}

It is an outstanding open problem whether a drop as exhibited
by our two-circle construction can occur for the uniform selection
from a \emph{convex} body: Can the expectation of the number of
vertices of a random polytope decrease in such a setting? This is
impossible in the plane~\cite{DGG} or for the three dimensional
ball~\cite{Beermann}, but open in general. See~\cite{BR} and the
discussion therein.

Perhaps Theorem~\ref{th:BL} remains valid for some restricted class of
measures~$\mu$, for instance, logconcave measures.  One approach to
circumvent the ``impossibility result'' of Theorem~\ref{th:example}
would be to first extend \eqref{eq:wet} and establish that for $c>1$
there is $c'$ such that for all $t>0$
\[ w^{\mu}(t) \le w^{\mu}(ct) \le c' \cdot w^{\mu}(t).\] 
The second step would derive from this property the extension of
Theorem~\ref{th:BL}. We don't know if any of these two steps is valid.

We can weaken the claim of Theorem~\ref{th:BL} in a different way,
while maintaining it for all measures.  For example, it is plausible
that the upper bound in the theorem holds for a subset of numbers
$n \in \mathbb{N}$ of positive density. On the other hand we do not
know if there is a measure for which the bound of Theorem~\ref{th:BL}
is valid only for a finite number of natural numbers.

\section{Proof of Theorem~\ref{th:BLgen}}\label{sec:proofBLgen}

Let $\mu$ be a probability measure in $\R^d$. For better readability
we drop all superscripts $\mu$.

\subsection{Lower bound}\label{p:lowerbound}

The proof of the lower bound is similar to the one in the
convex-body case. For every fixed point $x\in W_t$, by definition,
there exists a half-space $h$ with $x\in h$ and $\mu(h)\le t$.  If
$h\cap P_{n}$ is empty, then $x$ is not in $P_{n}$, and therefore, for $x\in W_t$, 
\begin{equation}\label{A.}
  \Pr[x\notin P_{n}] \ge
  \Pr[h\cap P_{n}=\emptyset] = (1-\mu(h))^{n}\ge
  (1-t)^{n}.
\end{equation}
Then, for any $t$,
\[\begin{aligned}
1-\E [\mu(P_{n})] & =  \int_{x\in \R^d} \Pr[x \notin P_{n}]d\mu(x) \\
 &  \ge \int_{x\in W_t} \Pr[x \notin P_{n}]d\mu(x)\\
 &  \ge    \int_{x\in W_t}(1-t)^{n} d\mu(x)  = (1-t)^{n} w(t).
\end{aligned}\]
We choose $t=1/n$. Since the sequence $(1-\frac1n)^{n}$ is increasing, for $n \ge 2$ we have $1-\E [\mu(P_{n})] \ge \frac14  w(\tfrac1n)$.\qed

\bigskip

To obtain the analogous lower bound from
Theorem~\ref{th:fvector}(i), we write
\[ \E\ptb{f_0(P_n)} = n\E\ptb{1-\mu(P_{n-1})} \ge n(1-t)^{n-1} w(t).\]
Again, choosing $t=1/n$ yields the claimed lower bound
\[ \E\ptb{f_0(P_n)} \ge
n\pth{1-\frac1n}^{n-1} w(\tfrac1n)\ge {\frac1e}nw(\tfrac1n),\]
since the sequence
$(1-\frac1n)^{n-1}$ is now decreasing to ${\frac1e}$.

\subsection{Floating bodies and $\eps$-nets}\label{s:epsnet}

Before we turn our attention to the upper bound, we will point out a
connection
to
$\eps$-nets. Consider a probability space $(U,\mu)$ and a
family $\hh$ of measurable subsets of $U$. An \emph{$\eps$-net} for
$(U, \mu, \hh)$ is a set $S\subseteq U$ that intersects every
$h\in \hh$ with $\mu(h)\ge \eps$~\cite[$\mathsection 10.2$]{Mat}. In
the special case where $U=(\R^d,\mu)$ and $\mathcal{H}$ consists of
all half-spaces, if a set $S$ is an $\eps$-net, then the convex hull
$P$ of $S$ contains $\R^d\setminus W_\eps$. Indeed, assume that
there exists a point $x$ in $\R^d\setminus W_\eps$ and not in
$P$. Consider a closed halfspace $h$ that contains $x$ and is disjoint
from $P$.  Since $x \notin W_\eps$ we must have $\mu(h) > \eps$
and $S$ cannot be an $\eps$-net.

We call the region $\R^d\setminus W_\eps$ the \emph{floating body}
of the measure $\mu$ with parameter $\eps$, by analogy to the case of
convex bodies. The relation between floating bodies and $\eps$-nets
was first observed by Van Vu, who used the $\eps$-net Theorem to prove
that $P_n^{\mu}$ contains $\R^d \setminus W_{c\log n/n}$ with
high probability~\cite[Lemma 4.2]{VanVu} (a fact previously
established by B\'ar\'any and Dalla~\cite{BD} when $\mu$ is the
normalized Lebesgue measure on a convex body). This implies that, with
high probability, $1-\mu(P_n)\le w(c\log n/n)$. The analysis we give in
Section~\ref{sec:upperbound} refines Vu's analysis to sharpen the
constant. Note that Theorem~\ref{th:example} shows that Vu's result is
already asymptotically best possible.

\subsection{Upper bound}\label{sec:upperbound}

For $d=1$, the proof of the upper bound is straightforward and may actually be improved. 
Indeed, we have $w(t)=\min\{2t,1\}$, and Efron's Formula~\eqref{eq:Efron2} yields
\[
\E\ptb{1-\mu(P_n)} \le \frac1{n+1} \E\ptb{f_0(P_{n+1})} \le \frac2{n+1} \le w\pth{\frac{1}{n+1}} \le w\pth{3 \frac{\ln n}{n}}.
\]
We will therefore assume $d\ge2$.

We use a lower bound on the probability of a
random sample of $U$ to be an $\eps$-net for $(U, \mu, \hh)$.
We define the \emph{shatter function} (or growth function) of the family
$\hh$ as
\begin{displaymath}
  \pi_{\hh}(N) = \max_{X\subseteq U, |X|\le N} |\{\, X\cap h :
  h \in \hh\,\}|,
\end{displaymath}

\begin{lemma} [{\cite[Theorem 3.2]{KPW}}]
\label{lemma-shatter} 
Let $(U,\mu)$ be a probability space and $\hh$ a family of measurable
subsets of $U$. Let $X_s$ be a sample of $s$ random independent
elements chosen according to $\mu$. For any integer $N> s$, the
probability that $X_s$ is not a $\eps$-net for $(U,\mu,\hh)$ is at most
\begin{equation}
  \nonumber 
 2\pi_{\mathcal H}(N) \cdot
  (1-\tfrac{s}{N})^{(N-s)\eps-1}.
\end{equation}
\end{lemma} 

Lemma~\ref{lemma-shatter} is a quantitative refinement of a
foundational result in learning theory ~\cite[Theorem~2]{VC}.
It is commonly used to prove that small $\eps$-nets exist for range
spaces of bounded Vapnik-Chervonenkis dimension~\cite{HW}, see
also~\cite[Theorem~3.1]{KPW} or~\cite[Theorem~15.5] {Pach}. For that
application, it is sufficient to show that the probability of failure
is less than 1; This works for $\eps\approx d \ln n/n$ (with
appropriate lower-order terms), where $d$ is the Vapnik-Chervonenkis
dimension.  In our proof, we will need a smaller failure probability
of order $o(1/n)$, and we will achieve this by setting $\eps\approx
(d+2) \ln n/n$.  We will apply the lemma in the case where $U = \R^d$
and $\hh$ is the set of halfspaces in~$\R^d$.  We mention that by
increasing $\eps$ more agressively, the probability of failure can be
made exponentially small.

For the family $\hh$ of halfspaces in $\R^d$, we have the
 following sharp bound on the shatter function~\cite{Harding}:
\[ \pi_{\hh}(N)\le 2\sum_{i=0}^{d} \binom{N-1}i.\]

The proof of the upper bound of Theorem~\ref{th:BLgen} starts
by remarking that for any $\eps \in [0,1]$ we have:
\begin{align*}
\E\ptb{1-\mu(P_n)} & = \int_{\R^d} \Pr[x \notin P_{n}]d\mu(x)\\
            & =  \int_{\R^d \setminus W_\eps} \Pr[x \notin P_{n}]d\mu(x) +  \int_{W_\eps} \Pr[x  \notin P_{n}]d\mu(x) \\
            & \le \int_{\R^d \setminus W_\eps} \Pr[\R^d \setminus W_\eps \not\subseteq P_{n}]d\mu(x) + \int_{W_\eps}d\mu(x)\\
            & \le \Pr[\R^d \setminus W_\eps \not\subseteq P_{n}] + w(\eps).
\end{align*}
Here, the first inequality between the probabilities holds since the
event $x \notin P_{n}$ trivially implies that $\R^d \setminus
W_\eps \not\subseteq P_{n}$ when $x\in \R^d \setminus W_\eps$.  We
thus have
\[\E\ptb{1-\mu(P_n)} \le  w(\eps) + \Pr[\R^d \setminus W_\eps \not\subseteq P_{n}].\]

We now want to set $\eps$ so that $\Pr[\R^d \setminus
W_\eps \not\subseteq P_{n}]$ is $\frac{\eps_d(n)}{n}$ with
$\eps_d(n) \to 0$ as $n \to \infty$. As shown in
Section~\ref{s:epsnet}, the event $\R^d \setminus W_\eps \not\subseteq
P_{n}$ implies that $P_{n}$ fails to be an $\eps$-net. The probability
can thus be bounded from above using Lemma~\ref{lemma-shatter} with
$s=n$. Taking logarithms, for any $N > n$,
\[ \ln \Pr[\R^d \setminus W_\eps \not\subseteq P_{n}] \le \ln
  \pi_{\mathcal H}(N) + \pth{(N-n)\eps-1}
  \ln (1-\tfrac{n}{N})
  + \ln 2.\]
Since we assume that $d\ge 2$, we have
\[ \pi_{\mathcal{H}}(N) \le 2\sum_{i=0}^{d} \binom{N-1}i\le N^d \quad \text{and} \quad  \ln\pi_{\mathcal{H}}(N) \le d \ln N.\]

We set $N = n \lceil \ln n \rceil$, so that:
\[\begin{aligned}
\ln \Pr[\R^d \setminus W_\eps \not\subseteq P_{n}] \le d \ln n & + d\ln \lceil \ln n \rceil\\& +
  \pth{(N-n)\eps-1}\ln (1-\tfrac{n}{N}) + \ln 2.
  \end{aligned}\]
We then set $\eps = \delta \frac{\ln n}{n}$, with $\delta \approx d$
to be fine-tuned later. If $n$ is large enough, the factor
$\pth{{(N-n)\eps-1}}\approx \delta\ln^2 n$ is nonnegative, and we can
use the inequality $\ln(1-x) \le -x$ for $x \in [0,1)$ in order to
bound the second term:
\begin{align*}
  \pth{(N-n))\eps-1} \ln (1-\tfrac{n}{N})  &\le- \bigl({(N-n)\eps-1}\bigr)\frac{n}{N}\\
  & = -n\eps +\frac{n}{\lceil \ln n \rceil}\eps +\frac1{\lceil \ln n
    \rceil} \le - \delta \ln n +\delta + 1.
\end{align*}
Altogether, we get
\[  \Pr[\R^d \setminus W_\eps \not\subseteq P_{n}] \le 2^{\delta+1}e \cdot n^{d-\delta}\lceil \ln n \rceil^d \]
so for every $\delta > d+1$ we have $\Pr[\R^d \setminus
W_\eps \not\subseteq P_{n}] = \frac{\eps_d(n)}{n}$ with $\eps_d(n) \to
0$ as $n \to \infty$. Setting $\delta=d+2$ yields the claimed
bound. \qed

\section{Proof of Theorem~\ref{th:example}}\label{sec:example}

In this section, logarithms are base $2$. For better readability we drop the superscripts $\nu$.

\subsection{The construction}

The measure $\nu$ is supported on a sequence of concentric circles
$C_1, C_2, \ldots$, where $C_i$ has radius \[r_i=1-\frac1{i+1}.\] On
each $C_i$, $\nu$ is uniform, implying that $\nu$ is rotationally
invariant. We let $D_i=\bigcup_{j \ge i} C_j$. For $i \ge 1$ we
put \[ \nu(D_i)= s_i := 4 \cdot 2^{-2^i} \] and remark that
$\nu(\R^2) = s_1 = 1$, so $\nu$ is a probability measure. The sequence
$\{s_i\}_{i \in \N}$ decreases very rapidly.  The probabilities of the
individual circles are
\[ p_i:=\nu(C_i)=s_i-s_{i+1}=4\cdot\pth{2^{-2^i}-2^{-2^{i+1}}} =s_i\pth{1-\frac{s_i}4}\approx s_i, \]
for $i\ge1$.

The infinite sequence of values $n$ for which we claim the inequality
of Theorem~\ref{th:example} is
\[  n_i :=2^{2^i+2i} \approx \tfrac 1{s_i} \log^2 \tfrac 1{s_i}.\]
In Section~\ref{wet}, we examine the wet part and prove that
$\wet{\tfrac{\log n_i}{n_i}}\le s_i$. We then want to establish the
complementary bound $\E\ptb{1-\nu(P_{n_i})}>
s_i/2$. Since $\nu$ is non-atomic, Efron's formula yields
\[ \E\ptb{1-\nu(P_{n_i })}
= \frac1{n_i+1} \E\ptb{f_0(P_{n_i+1)}}\]
and it suffices to establish that $\E\ptb{f_0(P_{n_i+1})} >
(n_i+1) s_i/2$. This is what we do in Section~\ref{random}.

\subsection{The wet part}
\label{wet}

Let us again drop the superscript $\nu$. Let $h_i$ be a closed
halfplane that has a single point in common with $C_i$, so its
bounding line is tangent to $C_i$. We have
\[ \wet{t}=s_{i} \text{, \ for $\nu(h_{i}) \le t < \nu(h_{i-1})$}.\]
So, as $t$ decreases, $\wet t$ drops step by step, each step being
from $s_i$ to $s_{i+1}$.  In particular,
\begin{equation}\label{eq:drops-one}
  \wet {t}\le s_{i} \iff  t< \nu(h_{i-1}).
\end{equation}

For $j > i$, the portion of $C_j$ contained in $h_i$ is equal to $2\arccos(r_i/r_j)$.
Hence, 
\[\nu(h_i\cap C_j)=\frac{\arccos(r_i/r_j)}{\pi} \cdot p_j.\]
We will bound the term $\arccos(r_i/r_j)$ by a more explicit
expression in terms of~$i$.  To get rid of the $\arccos$ function, we
use the fact that $\cos x \ge 1-x^2/2$ for all $x\in\R$. We obtain,
for $0\le y\le1$,
\[  \arccos (1-y)\ge \sqrt{2y}.\]
Moreover, the ratio ${r_i}/{r_{j}}$ can be bounded as follows:
\[\frac{r_i}{r_j}\le\frac{r_i}{r_{i+1}} =\frac i{i+1} \Bigm/\frac
{i+1}{i+2} = 1-\frac{1}{(i+1)^2}.\]
Thus we deduce that
\[ \frac{\arccos(r_i/r_j)}{\pi} \ge\frac{\arccos(1-
1/{(i+1)^2})}{\pi} \ge \frac{\sqrt{2}}{\pi(i+1)}. \]

We have established a bound on ${\arccos(r_i/r_j)}/{\pi}$, which is
the fraction of a single circle $C_j$ that is contained in
$h_i$. Hence, considering all circles $C_j$ with $j>i$ together, we
get
\[ \nu(h_i)\ge \frac{\sqrt{2}}{\pi(i+1)}s_{i+1}.\]
We check that for $i\ge 4$,
\[
  \frac{\log n_i}{n_i}=\frac{2^{i}+2i}{2^{2^i+2i}} =2^{-2^i}2^{-2i}
  (2^{i}+2i) =\frac{s_i}4 2^{-i}(1+2^{1-i}i)< s_i \frac{\sqrt{2}}{\pi
  i}\le \nu(h_{i-1}),
\]
because $2^{-i}(1+2^{1-i}i) < \frac{\sqrt{2}}{\pi i}$ for all $i\ge
4$.  Using \eqref{eq:drops-one}, this gives our desired bound:
\[\wet {\tfrac{\log n_i}{n_i}} \le s_i,\]
for all $i\ge 4$.  With little effort, one can show that actually
$\wet {\tfrac{\log n_i}{n_i}} = s_i$.  One can also see that, for any
$C>0$, the condition $\wet {C\tfrac{\log n_i}{n_i}}\le s_i $ holds if
$i$ is large enough, because the exponential factor $2^{-i}$ dominates
any constant factor~$C$ in the last chain of inequalities.  This
justifies the remark that we made after the statement of
Theorem~\ref{th:example}.

\subsection{The random polytope}
\label{random}

Assume now that $X_n$ is a set of $n$ points sampled independently
from $\nu$. We intend to bound from below the expectation
$\E\ptb{f_0(\conv X_{n_i+1})}$. Observe that for any $n
\in\N$ one has 
\[\E|X_n\cap C_i|=np_i\quad\text{and}\quad \Pr(X_n\cap
D_{i+1}=\emptyset)=(1-s_{i+1})^n.\]
Intuitively, as $n$ varies in the range near $n_i$, many points of
$X_n$ lie on $C_i$ and yet no point of $X_n$ lies in $D_{i+1}$. So
$P_n$ has, in expectation, at least $np_i \approx ns_i$ vertices. At
the same time, the term $\wet {\log n/n}$ in the claimed lower bound
drops to $s_i$. So the expected number of vertices is about $ns_i$
which is larger than $\frac 12 ns_i=\frac n2 \wet {\log n/n}$.

Formally, we estimate the expected number of vertices when $n=n_i+1$:
\begin{align*}
\E&\ptb{f_0(\conv X_{n_i+1})}\\
&\ge\E\ptb{f_0(\conv X_{n_i+1})\mid X_{n_i+1}\cap D_{i+1}=\emptyset} \cdot
\Pr(X_{n_i+1}\cap D_{i+1}=\emptyset)\\
&\ge\E[|X_{n_i+1}\cap C_i|]\cdot(1-s_{i+1})^{n_i+1}\\
&=(n_i+1)p_i(1-s_{i+1})^{n_i+1}\\
&=(n_i+1)s_i\left[ \frac {p_i}{s_i}(1-s_{i+1})^{n_i+1}\right]
\end{align*}
The last square bracket tends to 1 as $i \to \infty$. In particular, it is larger than~$\frac 12$ for $i \ge 4$. This shows that for all $i\ge 4$
$$
\E\ptb{f_0(\conv X_{n_i+1})}>\frac 12 (n_i+1)s_i
\ge
\frac 12(n_i+1)\wet{\tfrac{\log n_i} {n_i}}.
\eqno\qed\qedhere
$$

\subsection{Higher dimension}

 We can embed the plane containing $\nu$ in $\R^d$ for $d \ge 3$. The
analysis remains true but the random polytope is of course flat with
probability~1. To get a full-dimensional example, we can replace each
circle by a $(d-1)$-dimensional sphere, all other parameters being
kept identical: all spheres are centered in the same point, $C_i$ has
radius $1-\frac{1}{i+1}$, the measure is uniform on each $C_i$ and the
measure of $\cup_{j \ge i} C_i$ is $4 \cdot 2^{-2^i}$. The analysis holds
\emph{mutatis mutandis}.

As another example, which does not require new calculations, we can
combine $\nu$ with the uniform distribution on the edges of a regular
$(d-2)$-dimensional simplex in the $(d-2)$-dimensional subspace
orthogonal to the plane that contains the circles, mixing the two
distributions in the ratio $50:50$.

In all our constructions, the measure is concentrated on lower-dimensional
manifolds of $\R^d$, circles, spheres, or line segments. If a
continuous distribution is desired, one can replace each circle in the
plane by a narrow annulus and each sphere by a thin spherical shell,
without changing the characteristic behaviour.

\section{An alternative treatment of atomic measures}
\label{sec:alternative}
Even for measures with atoms, one can give a precise meaning to
Efron's formula:
The expression in~\eqref{eq:Efron0} counts the expected number
of convex hull vertices of $P_n$ that are unique in the sample $X_n$.
From this, it is obvious that 
Efron's formula~\eqref{eq:Efron} is a lower bound on $\E[f_0(P_n)]$
\eqref{eq:Efron2}.

For dealing with atomic measures, there is alternative possibility.
The resulting statements involve different quantities than our
original results, but they have the advantage of holding for every measure.
We denote
 by $\bar f_0(X_n
)$ the number of points of the sample $X_n$
that lie \emph{on the boundary}
of their convex hull $P_n
$, counted with multiplicity in case of coincident points.
We denote
 by $\breve P_n
$ the interior of 
 $P_n
$.
 Then a derivation analogous to \thetag{\ref{eq:Efron0}--\ref{eq:Efron}} leads to
the following variation of
Efron's formula: 
\begin{equation}
  \label{eq:Ef-open}
  \E [\bar f_0(X_n
)]
=n(1-\E [\mu(\breve P_{n-1}
)])
\end{equation}
We emphasize that we mean the boundary and interior with respect to
the ambient space $\R^d$, not the \emph{relative} boundary or interior.

Even for some non-atomic measures, this gives different results. Consider
the uniform distribution on the boundary of an equilateral triangle.
Then $\E[\bar f_0(X_n)]=n$,
while $\E[ f_0(P_n)]\le 6$.
Accordingly,
$\E [\mu(\breve P_{n})]=0$, while
$\E [\mu( P_{n})]$ converges to~$1$.

We denote the closure of
the wet part $W_t^{\mu}$ by
 $\bar W_t^{\mu}$
and 
its
measure by
$\bar w^{\mu}(t):=\mu(\bar W_t^{\mu})$. 

With these concepts, we can prove the following analogs of
Theorems~\ref{th:BLgen}--\ref{th:fvector}. Observe that for a measure $\mu$ for which for every hyperplane $H$, 
$\mu(H)=0$ the content of this theorem is the same as the previous ones. 

\begin{theorem}
  \begin{enumerate}
  \item [(i)]
  For any probability measure $\mu$ in $\R^d$ and $n \ge 2$,
  \begin{equation}
    \label{eq:closed-wet}
 \tfrac 14  \bar w^{\mu}(\tfrac1n) \le {\E[1-\mu(\breve P_n^\mu)]} \le 
\bar     w^{\mu}\bigl((d+2)\tfrac{\ln n}n\bigr) + \tfrac{\varepsilon_d(n)}{n},
  \end{equation}
and
  \begin{equation}
    \label{eq:boundary-f0}
     \tfrac 1e n \bar w^{\mu}(\tfrac1n) \le \E\ptb{\bar f_0(X_n^\mu)} \le 
     n \bar w^{\mu}\bigl({(d+2)\tfrac{\ln n}n}\bigr) + \varepsilon_d(n),
  \end{equation}
          where $\varepsilon_d(n)\to 0$ as $n\to+\infty$ and is independent of $\mu$.
        \item [(ii)] 
There is a non-atomic probability measure $\nu$ on $\R^2$ such
that
\[
    \E[1-\mu(\breve P_n^\nu)] > \tfrac 1{2} \cdot
\bar    w^\nu(\log_2 n/ n) \]
and
 \[
     \E\ptb{\bar f_0(X_n^\nu)} > \tfrac 1{2} n \cdot
    \bar w^\nu({\log_2 n/ n}) \]
 for infinitely many $n$.
\end{enumerate}

\end{theorem}

\begin{proof}[Proof sketch]
Since the derivation is parallel to the proofs in Sections
\ref{sec:proofBLgen}--\ref{sec:example}, we only sketch a few
crucial points.

(i)
For proving the lower bound in
\eqref{eq:closed-wet}, we modify the initial argument leading to \eqref{A.}:
For every
 fixed $x\in W_t$,
there is a \emph{closed} half-space $h$ with $x\in h$ 
whose 
 corresponding
 open halfspace $\breve h$
has measure
$\mu(\breve h)\le t$.
Therefore,
\begin{displaymath}
  \Pr[x\notin \breve P_{n}] \ge
  \Pr[h\cap \breve P_{n}=\emptyset]=
  \Pr[\breve h\cap P_{n}=\emptyset]
 = (1-\mu(\breve h))^{n}\ge
  (1-t)^{n}.
\end{displaymath}
The remainder of the proof can be adapted in a straightforward way.

In
Section~\ref{s:epsnet}, we have established that
for  an $\eps$-net $S$, its
 convex hull
$P$ 
contains $\R^d\setminus W_\eps$.
Since the interior operator is monotone, this implies that
 $\R^d\setminus \bar W_\eps \subseteq\breve P$.
Therefore, the 
 $\eps$-net argument of Section~\ref{sec:upperbound}
applies to the modified setting
and establishes
 the upper bound in~\eqref{eq:closed-wet}.

Finally,
by Efron's modified formula~\eqref{eq:Ef-open}, the result 
\eqref{eq:closed-wet}
carries over to
\eqref{eq:boundary-f0} as in our original derivation.

(ii)
The lower-bound construction of Theorem~\ref{th:example} gives zero
measure to every hyperplane, and therefore all quantities in part~(ii)
are equal to the corresponding quantites in Theorem~\ref{th:example}
and Theorem~\ref{th:fvector}(ii).
\end{proof}

\subsection*
{Acknowledgements.} I. B. was supported by the Hungarian National
Research, Development and Innovation Office NKFIH Grants K 111827 and
K 116769, and by the B\'ezout Labex (ANR-10-LABX-58). X.~G.\ was
supported by Institut Universitaire de France. 
The authors are grateful for the hospitality during the ASPAG (ANR-17-CE40-0017)
 workshop on geometry, probability, and algorithms in Arcachon in
 April 2018.

\vspace{.5cm} {\sc Imre B\'ar\'any}\\
  {\footnotesize R\'enyi Institute of Mathematics}\\[-1.3mm]
  {\footnotesize Hungarian Academy of Sciences}\\[-1.3mm]
  {\footnotesize PO Box 127, 1364 Budapest, and }\\[-1.3mm]
  {\footnotesize and}\\[-1.3mm]
  {\footnotesize Department of Mathematics}\\[-1.3mm]
  {\footnotesize University College London}\\[-1.3mm]
  {\footnotesize Gower Street, London WC1E 6BT}\\[-1.3mm]
  {\footnotesize England}\\[-1.3mm]
  {\footnotesize e-mail: {\tt barany@renyi.hu}}

\vspace{.5cm} {\sc Matthieu Fradelizi}\\
  {\footnotesize Universit\'e Paris-Est,}\\[-1.3mm]
  {\footnotesize LAMA (UMR 8050), UPEM, UPEC, CNRS, }\\[-1.3mm]
  {\footnotesize F-77454, Marne-la-Vall\'ee}\\[-1.3mm]
  {\footnotesize France }\\[-1.3mm]
  {\footnotesize e-mail: {\tt matthieu.fradelizi@u-pem.fr \tt }}

\vspace{.5cm} {\sc Xavier Goaoc}\\
  {\footnotesize Universit\'e de Lorraine, CNRS, Inria}\\[-1.3mm]
  {\footnotesize LORIA}\\[-1.3mm]
  {\footnotesize F-54000 Nancy}\\[-1.3mm]
  {\footnotesize France }\\[-1.3mm]
  {\footnotesize e-mail: {\tt xavier.goaoc@loria.fr}}

\vspace{.5cm} {\sc Alfredo Hubard}\\
  {\footnotesize Universit\'e Paris-Est, Marne-la-Vall\'ee }\\[-1.3mm]
  {\footnotesize Laboratoire d'Informatique Gaspard Monge }\\[-1.3mm]
  {\footnotesize 5 Boulevard Descartes, 77420 Champs-sur-Marne }\\[-1.3mm]
  {\footnotesize France }\\[-1.3mm]
  {\footnotesize e-mail: {\tt alfredo.hubard@u-pem.fr \tt }}

\vspace{.5cm} {\sc G\"unter Rote}\\
  {\footnotesize Freie Universit\"at Berlin}\\[-1.3mm]
  {\footnotesize Institut f\"ur Informatik}\\[-1.3mm]
  {\footnotesize Takustra\ss e 9, 14195 Berlin}\\[-1.3mm]
  {\footnotesize e-mail: {\tt rote@inf.fu-berlin.de}}

\end{document}